\documentclass[12pt,a4paper]{article}
\usepackage[biblatex]{anziamjedraft}
\usepackage{multirow}
\usepackage[table,xcdraw]{xcolor}

\usepackage{amsmath,amsthm,amssymb}
\usepackage{mathrsfs}
\numberwithin{equation}{section}
\usepackage{subcaption}
\usepackage{hyperref}

\newcommand{\ucex}{\tilde{u}_c}
\newcommand{\cex}{\tilde{C}}
\newcommand{\omex}{\tilde{\Omega}}

\newcommand{\alex}{\tilde{\alpha}}

\newcommand{\vecfd}{{\bf F}}

\def\delx{\ensuremath\dfrac{\partial}{\partial x}}

\title{Numerical solution of the two-phase tumour growth model with moving boundary}

\author{Gopikrishnan C. Remesan}
\address{IITB-Monash Research Academy,\,IIT Bombay,\,Powai,\,Mumbai 400076,\,INDIA}
\mailto{gopikrishnan.chirappurathuremesan@monash.edu}
\myorcid{0000-0002-4507-4463}

\date{\today}

\begin{document}

\maketitle

\begin{abstract}
A novel numerical technique has been proposed to solve a two-phase tumour growth model in one spatial dimension without needing to account for the boundary dynamics explicitly. The equivalence to the standard definition of a weak solution is proved. The method is tested against equations with analytically known solutions, to illustrate the advantages over the existing techniques. The tumour growth model is solved using the new procedure and showed to be consistent with results available in the literature. 
\end{abstract}

\tableofcontents

\section{Introduction}
\label{sec:intro}
We consider the tumour growth model presented in the seminal paper by Breward \emph{et al.}~\cite{breward_2002}. The partial differential equations are defined in a time-dependent one-dimensional spatial domain. Such systems generally account for higher spatial dimensional models reduced to a single spatial dimension by symmetry
arguments~\cite{byrne_2003}. In the current model, tumour cells and surrounding fluid medium are considered as two distinct, actively interacting phases. The cell phase is viscous with viscosity $\mu$, and the fluid phase is inviscid. 

Let $\Omega(t) = (0,\ell(t))$ be the interval representing the tumour where $\ell(t)$ is the tumour radius at time $t$.  Set $D_T := \cup_{0 < t <T}\{t\} \times \Omega(t)$ and $B_T := \partial D_T \backslash \left(\{T\} \times \Omega(T) \cup \{0\} \times \Omega(0) \cup [0,T]\times \{0\}\right)$ (Figure~\ref{fig:geometry}). $B_T$ is assumed to be of class $\mathscr{C}^1$~\cite[p.~627]{EvansPDE}. The model seeks the variables $\alpha,u_c$ and $C$ that denote volume fraction of the tumour cells, velocity of the tumour cells and oxygen tension, respectively such that the following hold in $D_T$:
\begin{subequations}
\label{system:nd_sys}
\begin{align}
\label{eqn:vol_fraction}
\dfrac{\partial \alpha}{\partial t} + \delx (u_c \alpha) &= \alpha f(\alpha,C), \\
\label{eqn:cel_velocity}
\dfrac{k u_c \alpha}{1 - \alpha} - \mu \dfrac{\partial}{\partial x} \left(\alpha \dfrac{\partial u_c}{\partial x}\right) &= -\dfrac{\partial}{\partial x} \left( \alpha \dfrac{\alpha - \alpha^{*}}{(1 - \alpha)^2} H(\alpha - \alpha_{min})\right), \\
\label{eqn:oxygen_tension}
\dfrac{\partial C}{\partial t} - \dfrac{\partial^2 C}{\partial x^2} &= -\dfrac{Q \alpha C}{1 + \hat{Q_1} C},
\end{align}
where $f(\alpha,C) = \frac{(1 + s_1)(1 - \alpha)C}{1 + s_1 C} - \frac{s_2 + s_3 C}{1 + s_4 C}$. The positive constants $k$ (drag coefficient) controls the drag between the phases; $\alpha^{\ast}$ and $\alpha_{\text{min}}$, the stress in the cellular phase; $s_1$, $s_2$, $s_3$ and $s_4$, birth and death rates. The Heaviside function $H(x) = 1$ if $x \geq 0$ and zero otherwise. The two-phase model in Breward \emph{et al.}~\cite{breward_2002} uses a quasi-steady state assumption for oxygen tension which is relaxed in this study. This means the explicit temporal variation of oxygen tension is considered which makes it a parabolic equation ~\eqref{eqn:oxygen_tension}.  The initial and boundary conditions are
\begin{gather}
\label{eqn:initial_cond}
\alpha(0,x) = \alpha_0(x),\;\;C(0,x) = C_0(x)\;\;\forall x \in \Omega(0), \\
\label{eqn:bdr_cond_1}
u_c(t,0) = 0,\;\mu \dfrac{\partial u_c}{\partial x}(t,\ell(t)) = \dfrac{\alpha(t,\ell(t))  -\alpha_{\text{min}}}{(1 - \alpha(t,\ell(t)))^2} H(\alpha(t,\ell(t)) - \alpha_{\text{min}}),\\
\label{eqn:bdr_cond_2}
\dfrac{\partial C}{\partial x}(t,0) = 0,\;C(t,\ell(t)) = 1 \quad \forall t \in (0,T), \\
\label{eqn:bd_velcoity}
\ell'(t) = u_c(t,\ell(t))\;\;\forall t,\;\; \ell(0) = \ell_0.
\end{gather}
Here, $\alpha_0(x)$ satisfies $0 < m_{\alpha} \leq \alpha_0(x) \leq M_{\alpha} < 1$ for every $x \in \Omega(0)$.
\label{system:nd_system}  
\end{subequations}

\begin{figure}
\centering
\includegraphics[scale=0.15]{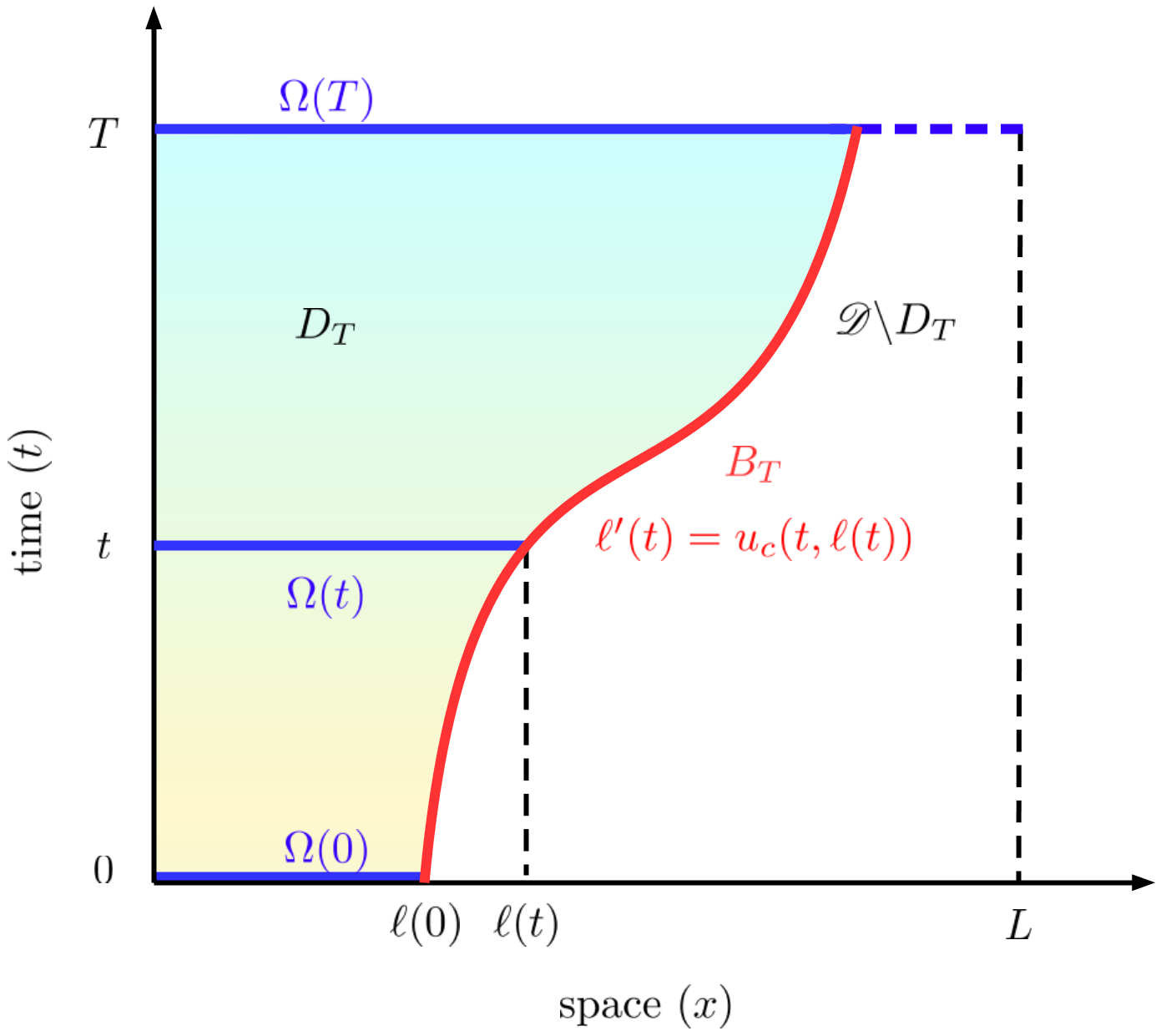}
\caption{The rectangle $(0,T) \times (0,L)$ is the time-independent domain $\mathscr{D}$. The region to the left of the red curve $B_T$ is $D_T$\,;\;to the right, $\mathscr{D}\backslash D_T$. }
\label{fig:geometry}
\end{figure}

The standard method to solve the system of equations of the form~\eqref{eqn:vol_fraction}-\eqref{eqn:bd_velcoity} is to transform the domain $\Omega(t)$ into a fixed interval using suitable change of
variables~\cite{breward_2003,ward_1,ward_2}. An inverse transform is then applied to obtain the solution in the moving domain. Even though this method is commonly adopted, it comes with significant drawbacks. 

Firstly, the change of variable is computable only in a few cases where the geometry of the problem is simple enough. This is even harder in 2D and 3D domains. Secondly, for the clear choice of $x \rightarrow \xi := x/\ell(t)$ in 1D case, the discretisation error is proportional to $\ell(t) \Delta \xi$. An alternative choice is to discretise $(0,\ell(t))$ to apply  numerical schemes. In this method re-meshing needs to be done at each time step which may become computationally expensive. 

In this article, a new numerical technique that overcomes theses disadvantages is introduced. These are done by presenting the notion of solutions on a larger domain which contains all the time-dependent domains $\Omega(t)$ for a finite time. This domain, referred to as the \emph{extended domain}, is time-independent and requires only one initial spatial discretisation; thereby avoiding the need to re-mesh. Also, the discretisation error becomes free from the dependence on $\ell(t)$.

This paper is organised as follows. In Section~\ref{sec:extended_model}, a novel method that is referred to as \emph{extended model} in the rest of the paper is introduced and its equivalence to the standard model is proved. In Section~\ref{sec:num_experiments}, the numerical technique is developed and the results are presented. The results are compared with a model problem for which analytic solutions are known. The effect of parameters in the new method is also investigated. The extended version is solved using the numerical technique developed and compared with results available from the literature. The paper ends with a conclusion in Section~\ref{sec:conclusion}.

\section{Extended model}
\label{sec:extended_model}
The notion of weak solutions in the given domain and extended domain are presented in this section. The solutions for the extended and original version are proved to be equivalent. For $p \in [1,\infty]$ and a family of domains $\{\Omega(t)\}_{0 \leq t \leq T}$, define

\begin{eqnarray*}
    \mathbb{L}^{p}\left(0,T;H^1(\Omega(t))\right) & := &\left\{v : [0,T]\times\mathbb{R} \rightarrow \mathbb{R} \;\middle\vert\; v(\cdot,t) \in H^1(\Omega(t)), \right. \\
    & &   \left. \forall t \in [0,T],\,\middle\vert\middle\vert\, ||v(\cdot,t)||_{H^1}\,\middle\vert\middle\vert_{L^p(0,T)} < \infty\right\}.
\end{eqnarray*}

  Multiply~\eqref{eqn:vol_fraction} by a test function $\phi \in \mathscr{C}_c^{\infty}\left(\overline{D_T}\backslash \left( \{T\} \times \Omega(T)\right)\right)$ and apply integration by parts. A use of~\eqref{eqn:bd_velcoity} and~\eqref{eqn:initial_cond} yields
\begin{gather}
\int_{D_T} \left( \alpha f(\alpha, C)\phi + (\alpha, u_c \alpha)\cdot \nabla_{t,x}\phi \right) \,\mathrm{d}t\,\mathrm{d}x +  \int_{\Omega(0)} \phi(0,x)\alpha_0(x)\,\mathrm{d}x  = 0. 
\label{eqn:alpha_weakform}
\end{gather}
This constitutes the weak formulation of the hyperbolic conservation law. The weak solutions of the problem in \ref{system:nd_system} and in the extended model are defined next.  Firstly, we give the definition of the solution in the domain $D_T$. \\

\begin{definition}[Weak solution I] 
By a weak solution of the system~\eqref{system:nd_system} in $D_T$ we mean a 4-tuple $(\alpha,u_c,C,\Omega)$ such that $0 < \overline{m}_{\alpha}  \leq \alpha \leq \overline{M}_{\alpha} < 1$, $C \geq 0$ and
\begin{itemize}
\item[1.]  $\alpha \in L^{\infty}(D_T)$ satisfies \eqref{eqn:alpha_weakform} for every $\phi \in \mathscr{C}_c^{\infty}\left(\overline{D_T}\backslash  \left( \{T\} \times \Omega(T) \right) \right)$.
\item[2.] $u_c \in \mathbb{L}^\infty(0,T;H^1(\Omega(t)))$ with $u\vert_{x=0} = 0$ and $C \in \mathbb{L}^2(0,T;H^1(\Omega(t)))$ with $C\vert_{x=\ell(t)} = 1$ are solutions of~\eqref{eqn:cel_velocity} and~\eqref{eqn:oxygen_tension} in the sense of distributions. 
\item[3.] The domain $\Omega(t)$ is the open interval $(0,\ell(t))$ where $\ell(t)$ is governed by~\eqref{eqn:bd_velcoity}.
\end{itemize}
\label{def:weak_1}
\end{definition}
The definition of the weak solution in the domain $\mathscr{D}$ is given next. $\mathscr{D}$ is the extended domain given by $(0,T) \times (0,L)$ (Figure~\ref{fig:geometry}) where $L$ is chosen such that $\ell(t) < L$ for every $t \leq T$. \\

\begin{definition}[Weak solution II] 
By a weak solution of the system~\eqref{system:nd_system} in the extended domain $\mathscr{D}$ (Figure~\ref{fig:geometry}) we mean a 4-tuple $(\alex,\ucex,\cex,\omex)$ such that $0 < \overline{m}_{\alpha} \leq \alex\vert_{\omex} \leq \overline{M}_{\alpha} < 1$, $\tilde{C} \geq 0$ and
\begin{itemize}
\item[1.] $\alex \in L^{\infty}(\mathscr{D})$ such that for every $\tilde{\phi} \in \mathscr{C}_c^{\infty} ([0,T)\times (0,L))$ 
\begin{equation}
\hspace{-1cm}\int_{\mathscr{D}} \left( \alex f(\alex,\cex)\tilde{\phi} + (\alex,\ucex \alex).\nabla_{t,x}\tilde{\phi}\right)\,\mathrm{d}t\,\mathrm{d}x + \int_{\Omega(0)} \tilde{\phi}(0,x)\alpha_{0}(x)\,\mathrm{d}x  = 0.
\label{eqn:weakform_alphaex}
\end{equation}
\item[2.] For a fixed $t$, $\omex(t) := \{ x : \alex(t,x) > 0 \}$, $\ucex = 0$, $\cex = 1$ on $(0,L) \backslash \omex(t)$. Define $\tilde{D}_T := \cup_{0 < t <T}\{t\} \times \omex(t)$.
\item[3.] $\ucex \in L^\infty(\mathscr{D})$ with $u_c := \ucex\vert_{\tilde{D}_T }$, $u_c \in \mathbb{L}^\infty(0,T;H^1(\tilde{\Omega}(t)))$ and, $\cex \in L^2(\mathscr{D})$ such that $C := \cex\vert_{\tilde{D}_T} \in \mathbb{L}^2( 0,T;H^1(\omex(t)))$ are solutions of~\eqref{eqn:cel_velocity} and~\eqref{eqn:oxygen_tension} in the sense of distributions.
\end{itemize}
\label{def:weak_2}
\end{definition}

\begin{theorem}
If $(\alpha,u_c,C,\Omega)$ is a weak solution I, then $(\alex,\ucex,\cex,\omex)$ defined by $\alex := \alpha, \ucex := u_c$ and $\cex := C$ in $D_T$ and, $\alex := 0,\ucex := 0,\cex := 1$ in $\mathscr{D}\backslash D_T$ with $\omex(t) := \Omega(t)$ is weak solution II. Conversely, if $(\alex,\ucex,\cex,\omex)$ is a weak solution II, then $(\alpha,u_c,C,\Omega)$ with $\Omega = \omex$ and $\alpha := \alex\vert_{\tilde{D}_T},u_c := \ucex\vert_{\tilde{D}_T}$ and $C := \cex\vert_{\tilde{D}_T}$ is a weak solution I. 
\label{thm:equivalence_thm}
\end{theorem}
\begin{proof}
Let $(\alpha,u_c,C,\Omega)$ be a weak solution I and $\tilde{\phi} \in \mathscr{C}_c^{\infty} ([0,T)\times (0,L))$. Since $\tilde{\phi}\vert_{D_T} \in \mathscr{C}_c^{\infty}\left(\overline{D_T}\backslash  \left( \{T\} \times \Omega(T)\right)\right)$, ~\eqref{eqn:alpha_weakform} holds true. Let $\alex = \alpha$ in $D_T$ and $\alex = 0$ in $\mathscr{D}\backslash D_T$. Then a use of the definitions of $\ucex$ and $\cex$ yields 
\begin{align}
\label{eqn:eqweakform_dt}
\int_{D_T} \left( \alex f(\alex, \cex)\tilde{\phi} + (\alex, \ucex \alex)\cdot \nabla_{t,x}\tilde{\phi}\right) \,\mathrm{d}t\,\mathrm{d}x + \int_{\Omega(0)} \tilde{\phi}(0,x)\alpha_0(x)\,\mathrm{d}x  &= 0, \\
\label{eqn:eqweakform_dtex}
\int_{\mathscr{D}\backslash D_T} (\alex,\ucex\alex)\cdot \nabla_{t,x} \tilde{\phi}\,\,\mathrm{d}t\,\mathrm{d}x + \int_{\mathscr{D}\backslash D_T} \alex f(\alex,\cex)\tilde{\phi}\,\mathrm{d}t\,\mathrm{d}x &= 0.
\end{align}
Add \eqref{eqn:eqweakform_dt} and \eqref{eqn:eqweakform_dtex} to obtain
\begin{equation}
\int_{\mathscr{D}}\left(  \alex f(\alex, \cex)\tilde{\phi} +  (\alex, \ucex \alex)\cdot \nabla_{t,x}\tilde{\phi}\right) \,\mathrm{d}t\,\mathrm{d}x + \int_{\Omega(0)} \tilde{\phi}(0,x)\alpha_0(x)\,\mathrm{d}x  = 0.
\end{equation}
Therefore \eqref{eqn:weakform_alphaex} holds true. The conditions on $\ucex$ and $\cex$ follow naturally from the definition~\ref{def:weak_2}.  Since $\alex > 0$ in $D_T$ and $\alex = 0$ in $\mathscr{D}\backslash D_T$, $\omex(t) = \Omega(t)$ for every $t \in [0,T)$. Therefore $(\alex,\ucex,\cex,\omex)$ is a weak solution II.

Conversely, assume that $(\alex,\ucex,\cex,\omex)$ is a weak solution II. Let $\phi \in \mathscr{C}_c^{\infty}\left(\overline{D_T}\backslash  \left( \{T\} \times \Omega(T)\right)\right)$. Define $\tilde{\phi} \in \mathscr{C}_c^{\infty} ([0,T)\times (0,L))$ such that $\tilde{\phi} = \phi$ in $D_T$. Since $\Omega(t) = \omex(t)$ for every $t$, $\alex = 0$ in $\mathscr{D}\backslash D_T$. Using this in~\eqref{eqn:weakform_alphaex} we obtain~\eqref{eqn:alpha_weakform}. We shall recover~\eqref{eqn:bd_velcoity} next. For this define a vector field $\mathbf{F} : \mathscr{D} \rightarrow \mathbb{R}^2$ by $\mathbf{F}(t,x) := (\alex, \ucex \alex).$  We set $\vecfd\vert_{B_T^{+}} = (\vecfd\vert_{D_{T}})\vert_{B_T}$ and $\vecfd\vert_{B_T^{-}} = (\vecfd\vert_{\mathscr{D}\backslash D_{T}})\vert_{B_T}$. Since the weak divergence of the vector field $\vecfd$ is -$\alex f(\alex,\cex) \in L^2(\mathscr{D})$, the flux of $\vecfd$ is continuous across $B_T$. Since $\alex = 0$ in $\mathscr{D}\backslash D_T$, $\vecfd\vert_{B_T^{-}} = {\bf 0}$. Therefore, $(\vecfd\vert_{B_T^{+}} - \vecfd\vert_{B_T^{-}})\cdot{\bf n}_{B_T} = (\alpha,u_c\alpha)\cdot {\bf n}_{B_T} = 0$ where ${\bf n}_{B_T}$ is the normal to $B_T$ given by $\left(\left\vert \ell'(t) \right\vert^2 + 1 \right)^{-1/2}\left(-\ell'(t),1\right)$. This gives $(\alpha,u_c\alpha)\cdot {\bf n}_{B_T} = 0$. Since $\alpha > 0$, $\ell'(t) = u_c(t,\ell(t))$. The conditions on $u_c$ and $C$ follows directly from the definitions.  Therefore $(\alpha,u_c,C,\Omega)$  is a weak solution I. 

This completes the proof of the equivalence between the solutions.
\end{proof}
\section{Numerical experiments}
\label{sec:num_experiments}
By Theorem~\ref{thm:equivalence_thm} it is enough to solve~\eqref{eqn:vol_fraction} in the extended domain $(0,L)$.
Equation~\eqref{eqn:vol_fraction} is solved using cell-centred finite volume methods. In particular, we use two methods to solve the volume fraction equation: upwinding with Godunov flux~\cite[p.~135]{eymard} (method U), and MUSCL with Godunov flux~\cite[p.~146]{eymard} (method M). The uniform space and time discretisations are 
$0 = x_0 < x_1 < \cdots < x_i < \cdots < x_M = L$, $0 = t_0 < t_2 < \cdots < t_j < \cdots < t_N = T$ with $h = x_{i+1} - x_{i}$ and $\Delta t = t_{j+1} - t_{j}$. The right hand side boundary $\ell(t)$ is approximated by $\ell_{h}(t) = \min_{x} \{ x :  \alex< \alpha_{\text{thr}}\;\text{ on } (x,L)\}$ where $\alpha_{\text{thr}}$ is a small positive number. Define $\tilde{\alpha} := 0$  for $x \geq \ell_{h}(t)$ to eliminate the error caused by small positive values of $\tilde{\alpha}$  (created by numerical diffusion) in there. Equations~\eqref{eqn:cel_velocity} and~\eqref{eqn:oxygen_tension} are solved using conforming $P_1$ finite element method (FEM) in space and forward finite difference in time, in the reconstructed domain $(0,\ell_{h}(t))$. This procedure, referred to as scheme A in the rest of the paper, is outlined below. 

\begin{itemize}
\item[1.] Start at $t_0 = 0$. Solve $\tilde{u}_{ch}^0$ using $\alex_{h}^0$ (initial condition).
\end{itemize}
For $j = 1$ to $N$,\;$t_j = t_{j-1} + \Delta t$.
\begin{itemize}
\item[2.] $\ell_{h}^j = \displaystyle\min\limits_{x_i} \{ x_i :  \alex_h^{j-1} < \alpha_{\text{thr}}\;\text{ on } (x_i,L)\}$. 
\item[3.] Find $\tilde{u}_{ch}^j\;\&\;\tilde{C}_h^j$ in $(0,\ell_{h}^j)$ ($P_1$ conforming FEM).
\item[4.] Extrapolate $\tilde{u}_{ch}^j = 0\;\&\;\tilde{C}_h^j = 1$ to $(\ell_{h}^j,L)$.
\item[5.] Find $\alex_h^{j}$ in $(0,L)$ (method U or M).
\end{itemize}

The complete elimination of re-meshing and applicability in higher dimensions are the major advantages of this scheme.  Scheme B denotes the procedure of obtaining numerical solution in the scaled domain $(0,1)$~\cite{breward_2002}. Two test cases are considered in the numerical experiments.

In the first case, the cell velocity $u_c$ and the oxygen tension $C$ are assumed to be unity. In this case \eqref{eqn:vol_fraction} reduces to a semi-linear advection equation which can be solved analytically by the method of characteristics. The analytical solution is compared with the numerical solutions. We also study the influence of $\alpha_{\text{thr}}$ on locating the tumour frontier. In the second case, we compare the approximate solutions of the full system \eqref{eqn:vol_fraction}-\eqref{eqn:bd_velcoity}.

In all numerical tests the values of the parameters are set to be $s_1 = 10 = s_4,\,s_2 = 0.5 = s_3$, $k = 1 = \mu$, $Q = 0.5$, $\hat{Q}_1 = 0$ and $\ell(0) = 1$ to preserve conformity with Breward \emph{et al}.~\cite{breward_2002}.
\subsection*{Case 1}
The analytical solution to~\eqref{eqn:vol_fraction} in the case where $C = u_c = 1$ is:
\begin{equation}
\alpha(t,x) = \dfrac{(c_2 - c_1) \alpha_0(x - t)\exp((c_1 - c_2)t)}{c_1 \alpha_0(x - t) \left(1 -\exp((c_1 - c_2)t)\right) + c_2 - c_1}
\end{equation}
where $c_1 = 1$ and $c_2 = \frac{s_2 + s_3}{1 + s_4}$. The initial data considered are~(i) $\alpha_0(t,x) = 0.5\left(0.02 + \cos^2\left(x\right)\right) \chi_{[0,1]}$~(ii) $\alpha_0(t,x) = 0.5\left(0.02 + \sin^2\left(x\right)\right) \chi_{[0,1]}$ and~(iii) $\alpha_0(t,x) = \frac{\chi_{[0,1]}}{2}\frac{1 + \exp(x - 0.5)^2}{1 + \exp(2(x - 0.5)^2)}$, where $\chi_{[0,1]} = 1$ in $[0,1]$ and $0$ otherwise. Here $T=5,\,L = 6,\,\Delta t = 0.01,\,\Delta x = 0.02,\,\alpha_{\text{thr}} = 0.04$ (method U) and $\alpha_{\text{thr}} = 0.004$ (method M). The reduction in the numerical diffusion from former to latter method explains the reduction in the threshold value.

\begin{figure}
\centering
\includegraphics[scale=0.07]{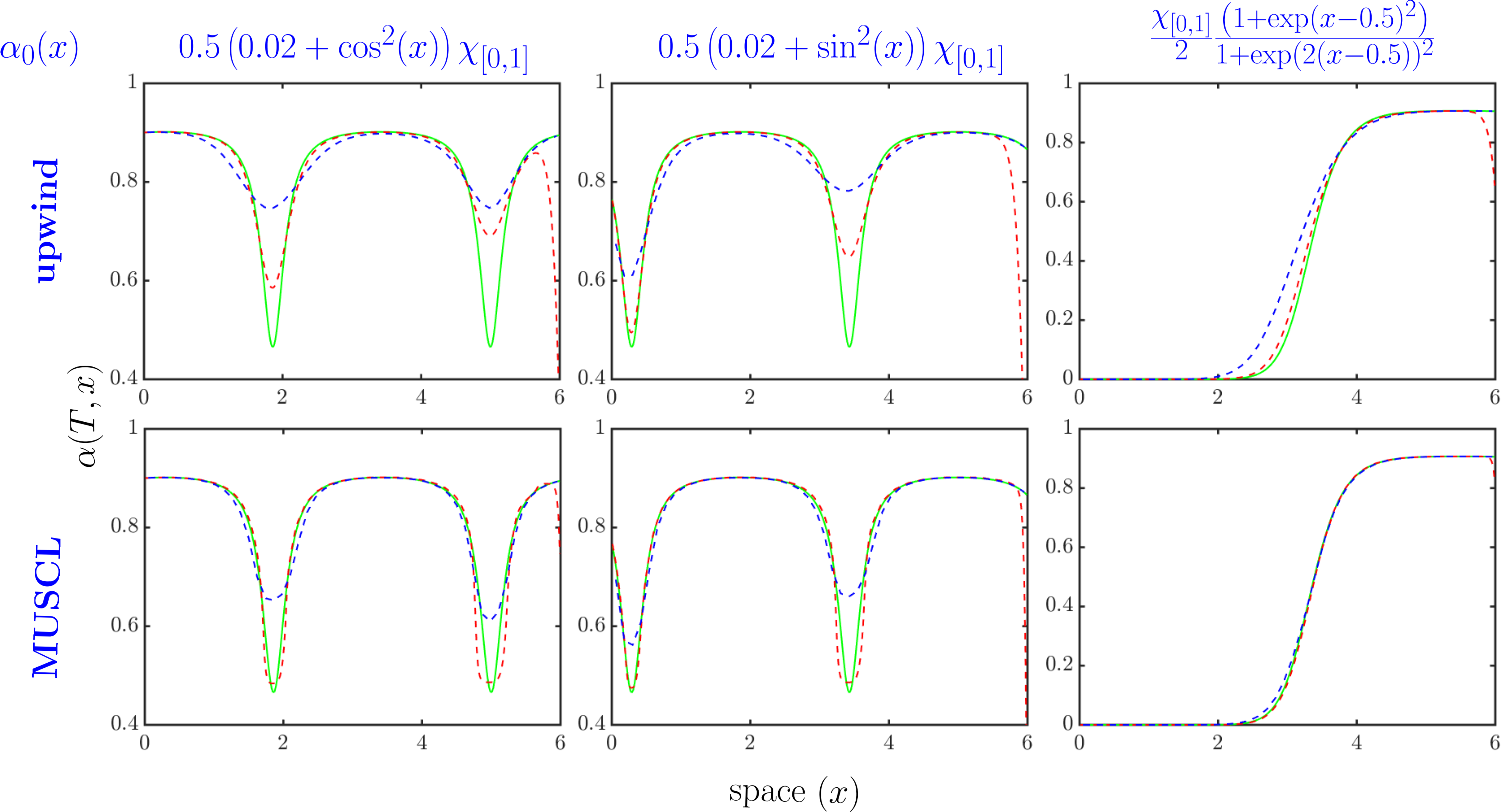}
\caption{Case 1: scheme A - red, scheme B - blue, analytical solution - green. Upper row - method U, lower row - method M.}
\label{figure:case1}
\end{figure}

Approximate solution obtained in the extended domain captures the properties of the analytical solution better than the one obtained in the scaled domain (Figure~\ref{figure:case1}), though it is less accurate towards the discontinuity at $\ell_h^j$ in method U owing to high diffusion. But method M overcomes this disadvantage; the extended solution agrees well with the scaled solution towards the discontinuity, and remarkably better in the interior region (Figure~\ref{figure:case1}). The recovered radius, on the other hand, is in excellent agreement with the exact radius for both method M and U with a proper choice of the threshold value  (Table~\ref{table:radius_threshold_muscl}).

We conclude this section by analysing the dependency of the recovered radius on the threshold value $\alpha_{\text{thr}}$ for the MUSCL method. The relative error, $\Delta \ell_h  = \frac{|\ell(T) - \ell_h^N|}{\ell(T)}$, at $T=5$ is used as a quantification of the error in the recovered radius. Two sets of experiments are conducted; (a) varying $\alpha_{\text{thr}}$ at a fixed $\Delta x$ (b) varying $\Delta x$ at a fixed $\alpha_{\text{thr}}$. Table~\ref{table:radius_threshold_muscl} shows that there exists a wide range of $\alpha_{\text{thr}}$ and $\Delta x$ for which the error remains below 1$\operatorname{E}-2$. This assures the accuracy of the method while the selection of $\alpha_{\text{thr}}$ remains a pertinent problem. 
\begin{table}[]
\centering
\begin{tabular}{|c||c|c|c|c|c|}
\hline
\multirow{2}{*}{\textbf{\color{blue}$\Delta x$}} & \multicolumn{5}{c|}{\color{red} $\alpha_{\text{thr}}$} \\ \cline{2-6} 
                    & $\textcolor{red}{0.01}$  & $\textcolor{red}{0.008}$  & $\textcolor{red}{0.006}$  & $\textcolor{red}{0.004}$  & $\textcolor{red}{0.002}$ \\ \hline \hline 
\textcolor{blue}{0.01}                & 1.67$\operatorname{E}-3$  & 1.67$\operatorname{E}-3$   & 1.67$\operatorname{E}-3$   & 1.67$\operatorname{E}-3$   & 5.00$\operatorname{E}-3$ \\ \hline
\textcolor{blue}{0.02}                & 3.33$\operatorname{E}-3$  & 3.33$\operatorname{E}-3$  & 6.67$\operatorname{E}-3$   & 1.33$\operatorname{E}-2$ & 2.00$\operatorname{E}-2$  \\ \hline
\textcolor{blue}{0.04}             & 6.67$\operatorname{E}-3$  & 6.67$\operatorname{E}-3$   & 2.00$\operatorname{E}-2$  & 2.67$\operatorname{E}-2$   & 4.00$\operatorname{E}-2$ \\ \hline
\textcolor{blue}{0.06}                & 4.31$\operatorname{E}-3$ & 1.58$\operatorname{E}-2$   & 2.59$\operatorname{E}-2$   & 4.60$\operatorname{E}-2$   & 6.61$\operatorname{E}-2$  \\ \hline
\textcolor{blue}{0.08}                & 2.10$\operatorname{E}-2$   & 7.66$\operatorname{E}-3$   & 1.92$\operatorname{E}-2$   & 3.26$\operatorname{E}-2$   & 5.93$\operatorname{E}-2$  \\ \hline
\textcolor{blue}{0.1}                 & 3.33$\operatorname{E}-2$  & 1.67$\operatorname{E}-2$    & 1.67$\operatorname{E}-2$   & 5.00$\operatorname{E}-2$   & 8.33$\operatorname{E}-2$  \\ \hline
\end{tabular}
\caption{$\Delta \ell_h$ for case 1, method M.}
\label{table:radius_threshold_muscl}
\end{table}
\begin{table}[]
\centering
\begin{tabular}{|c||c|c|c|c|}
\hline
{\color[HTML]{3166FF}}          & \multicolumn{4}{c|}{{\color[HTML]{FE0000} \textbf{$\alpha_{\text{thr}}$}}}                                                     \\ \cline{2-5} 
\multirow{-2}{*}{{\color{blue} \textbf{$\Delta x$}}} & {\color{red} $0.04$} & {\color{red} $0.03$} & {\color{red} $0.02$} & {\color{red} $0.01$} \\ \hline \hline
{\color{blue} 0.01}                                  & 3.33$\operatorname{E}-3$                        & 3.33$\operatorname{E}-3$                        & 1.66$\operatorname{E}-2$                        & 3.83$\operatorname{E}-2$                       \\ \hline
{\color{blue} 0.02}                                  & 3.33$\operatorname{E}-2$                        & 3.33$\operatorname{E}-3$                        & 1.33$\operatorname{E}-2$                        & 5.68$\operatorname{E}-2$                       \\ \hline
{\color{blue} 0.04}                                  & 1.20$\operatorname{E}-1$                       & 7.33$\operatorname{E}-2$                        & 6.66$\operatorname{E}-3$                        & 6.00$\operatorname{E}-2$                        \\ \hline
\end{tabular}
\caption{$\Delta \ell_h$ for case 1, method U.}
\label{table:radius_threshold_upwind}
\end{table}
The range of $\alpha_{\text{thr}}$ and $\Delta x$ for the error remains low is thin for method U, which is expected considering the high numerical diffusion associated with it.
\subsection*{Case 2}
\begin{figure}
\centering
\includegraphics[scale=0.07]{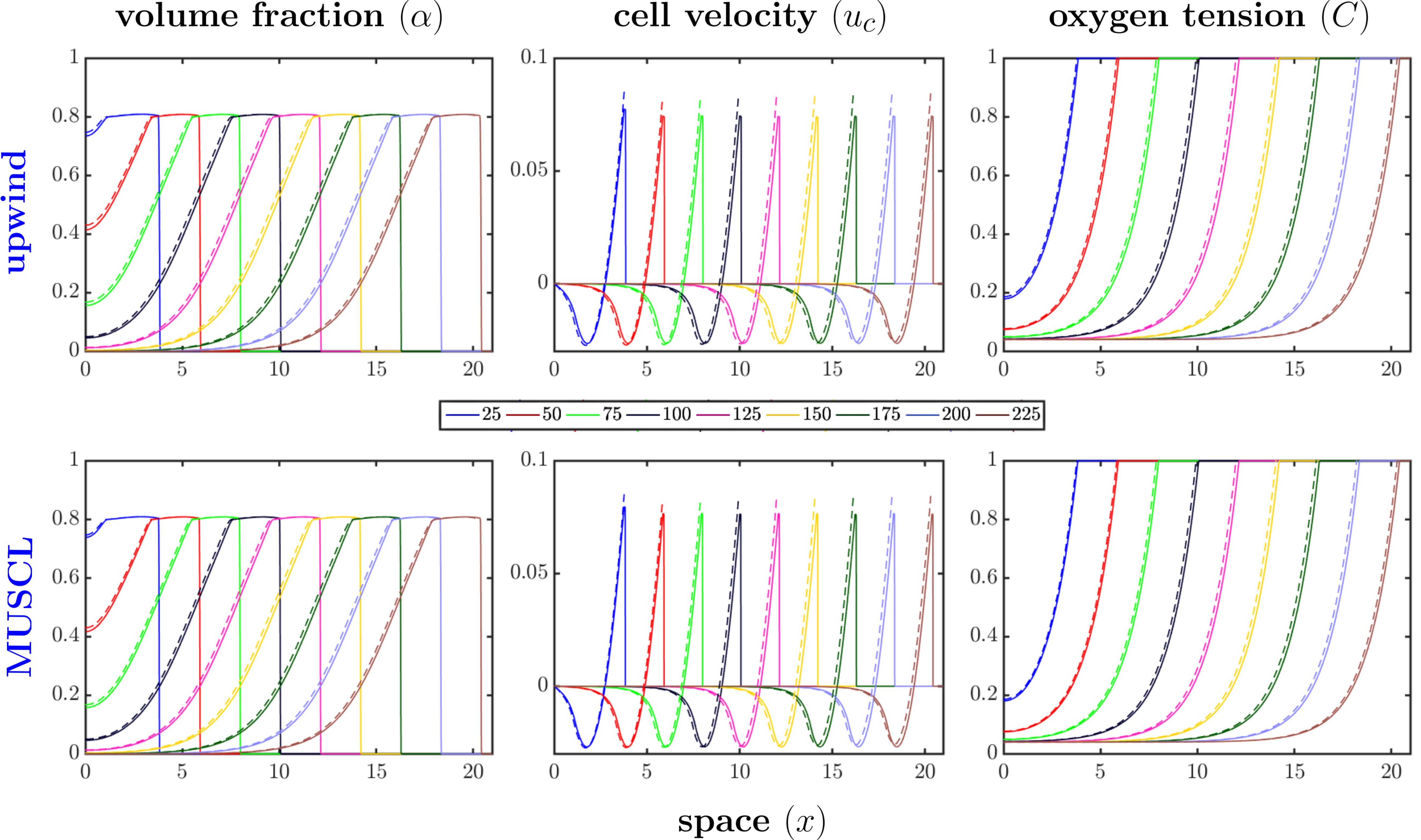}
\caption{Case 2: scheme A - solid lines, scheme B - dotted lines. Upper row - method U, lower row - method M. Each curve represents the variation of the corresponding variable with respect to space at fixed times $t=25,\,50,\cdots,225$.}
\label{fig:breward}
\end{figure}
The two phase model with all the system variables treated as unknowns is considered in this case. The parameters are chosen as $\Delta t = 0.01$, $\Delta x = 0.01$, $T = 228$, $L = 25$, $\alpha_{\text{thr}} = 0.004$ (method M) and $\alpha_{\text{thr}} = 0.01$ (method U) based on Tables~\ref{table:radius_threshold_muscl},~\ref{table:radius_threshold_upwind}. The initial condition is $\alpha_0(x) = 0.8$ for $0 \leq x\leq 1$ and $0$ otherwise. The moving boundary $\ell(t)$ is well captured by methods U and M. Since the exact value of $\ell(T)$ is not available, the error is quantified as the relative difference between scheme B and scheme A. The difference for method U is $6.18\operatorname{E}-3$ and method M is $5.69\operatorname{E}-3$. The numerical solution in the extended domain is in good agreement with the solution obtained from the scaled domain \cite{breward_2002}.

\section{Conclusion}
\label{sec:conclusion}
A novel numerical technique is developed to solve the two phase tumour growth problem and is tested against problems for which analytical solutions are known. For a fixed spatial mesh size the new method gives better solution than the standard method of solving in a scaled domain. The moving boundary is recovered from the numerical solution by comparing with a threshold value. It is found that an appreciable range of threshold values can be used along with higher order methods like MUSCL so that the error in the recovered radius can be kept low. The solution obtained by applying this technique shows very good agreement with solutions obtained using standard methods. This emphasises the reliability of the new method in extending it to solve tumour growth problems in higher dimensions while not solving for the boundary explicitly. 

\paragraph{Acknowledgement}
The author expresses gratitude towards A/Prof. J\'{e}r\^{o}me Droniou (Monash Univeristy), Dr. Jennifer Anne Flegg (Melbourne University) and Prof. Neela Nataraj (I.I.T. Bombay) for the valuable suggestions and help. 

    \bibliographystyle{plain}
    \bibliography{bibexport}

\end{document}